\documentclass[]{article}

\pdfoutput=1

\usepackage{amsfonts}
\usepackage{amsmath}
\usepackage{amsthm}
\usepackage{amssymb}
\usepackage{graphicx}
\usepackage{epstopdf}
\usepackage{verbatim}
\usepackage{xfrac}
\theoremstyle{plain}
\newtheorem{theorem}{Theorem}[section]
\newtheorem{lemma}[theorem]{Lemma}
\newtheorem{cor}[theorem]{Corollary}
\newtheorem{prop}[theorem]{Proposition}

\theoremstyle{definition}
\newtheorem{defn}[theorem]{Definition}

\DeclareMathOperator{\End}{End}

\DeclareMathOperator{\Sym}{Sym}
\DeclareMathOperator{\Id}{Id}

\DeclareMathOperator{\Hess}{Hess}
\DeclareMathOperator{\adj}{adj}

\title{Global properties of toric nearly K\"ahler manifolds}
\author{Kael Dixon}

\begin{document}

\maketitle

\begin{abstract}
	We study toric nearly K\"ahler manifolds, extending the work of Moroianu and Nagy. We give a description of the global geometry using multi-moment maps. We then investigate polynomial and radial solutions to the toric nearly K\"ahler equation.
\end{abstract}
%\tableofcontents
\section{Introduction}
	A nearly K\"ahler manifold is an almost Hermitian manifold $(M,g,J)$ such that $\nabla J$ is skew symmetric: $(\nabla_X J)X =0$ for every vector field $X$ on $M$. Each of these can be decomposed as a Riemannian product of nearly K\"ahler manifolds which are either K\"ahler, $6$-dimensional, homogeneous, or twistor spaces over quaternionic K\"ahler manifolds of positive scalar curvature \cite{nagy2002nearly}. We will focus on the case of $6$-dimensional nearly K\"ahler manifolds that are \emph{strict} in the sense that they are not K\"ahler. These are characterized by being the links of metric cones with holonomy $G_2$, which makes them Einstein with positive scalar curvature \cite{bar1993real}.
	
	A main challenge is to construct complete examples of $6$-dimensional strictly nearly K\"ahler manifolds (which will be referred to simply as nearly K\"ahler manifolds in the rest of the paper). There are exactly four homogeneous examples \cite{butruille2010homogeneous}: $\mathbb S^6$, $\mathbb S^3\times\mathbb S^3$, $\mathbb{CP}^3$, and the flag manifold $\mathrm{SU}_3/\mathbb T^2$. In \cite{foscolo2017new}, cohomogeneity one examples are constructed on $\mathbb S^6$ and $\mathbb S^3\times\mathbb S^3$. No other complete examples are known. The cohomogeneity two case has been studied in \cite{madnick2017nearly}, which shows that the infinitesimal symmetry group must be $\mathfrak u(2)$.
	
	We will skip to cohomogeneity three in exchange for having an abelian symmetry group by studying nearly K\"ahler manifolds which are toric in the sense that the automorphism group contains a $3$-torus. The homogeneous nearly K\"ahler structure on $\mathbb S^3\times\mathbb S^3$ is the only known example. The general case has been studied in \cite{moroianu2018toric}, where the local theory is shown to be equivalent to a Monge-Amp\`ere type equation which we will refer to as the toric nearly K\"ahler equation.  
	% (\ref{eqnNearlyToric}). 
	This paper represents the author's efforts to build on this work. The main result gives a global description of the orbit structure in the complete case:
	
	\begin{theorem}\label{thmMain}
		Let $M$ be a complete toric nearly K\"ahler manifold with the action of a torus $\mathbb T^3$ with Lie algebra $\mathfrak t$. Then $M/\mathbb T$ is homeomorphic to $\mathbb S^3$. In particular, the multi-moment maps $$(\mu,\varepsilon):M\to\Lambda^2\mathfrak t^*\oplus\Lambda^3\mathfrak t^*\cong\mathbb R^4$$ induce an injection $$(\bar\mu,\bar\varepsilon):M/\mathbb T\hookrightarrow \Lambda^2\mathfrak t^*\oplus\Lambda^3\mathfrak t^*\cong\mathbb R^4,$$ whose image is homeomorphic to $\mathbb S^3$. The $\mathbb T$ action is free away from a finite number orbits in $\ker\varepsilon$. Moreover the two orbits in $\ker\mu$ are Lagrangian.
	\end{theorem}
	
	This theorem generalizes previous work by the author in \cite{dixon2019multi}, which treats only the case of the homogeneous nearly K\"ahler structure on $\mathbb S^3\times\mathbb S^3$. In that case, $\mu(\ker\varepsilon)$ is Cayley's nodal cubic surface, whose $4$ nodal singular points correspond to the singular $\mathbb T$ orbits. 
	By studying the topological consequences of this theorem, we prove the following:
	
	\begin{cor}\label{corNumOrbits}
		Any complete toric nearly K\"ahler manifold has at least $4$ torus orbits where the action is not free. 
	\end{cor}	
As a consequence of this, radial solutions to the toric nearly K\"ahler equation cannot give complete metrics. By studying the corresponding ODE, we see that the singularity that forms must occur at the Lagrangian orbit.
	
	We also study the case when a hypothetical solution to the toric nearly K\"ahler equation is polynomial in the natural multi-moment map coordinates. The homogeneous nearly K\"ahler structure on $\mathbb S^3\times\mathbb S^3$ corresponds to a cubic solution $\varphi_0$ shown in equation (\ref{eqnCubicSoln}). Using an old theorem of Hesse \cite{lossen2004does}, we prove:
	\begin{theorem}\label{thmPolySol}
		Every polynomial solution of the toric nearly K\"ahler equation with degree at most $5$ is equivalent to the cubic solution $\varphi_0$ up to coordinate transformation.
	\end{theorem}	
	The toric nearly K\"ahler equation restricted to the space of polynomials is overdetermined for polynomials of degree greater than three, so it is unlikely that there will be other polynomial solutions. However, to show this explicitly is computationally difficult even in the quintic case. 

\subsection{Acknowledgements}
	This research was funded by the Simons Collaboration on Special Holonomy in Geometry and Physics (\#488635
Simon Salamon). I am very grateful for many interesting and productive discussions about this work with others, especially Bobby Acharya, Lorenzo Foscolo, Andrei Moroianu, and Simon Salamon.

\section{Local theory}\label{secLoc}

In this section we review the local theory of toric nearly K\"ahler manifolds from \cite{moroianu2018toric}, although we will use a coordinate invariant treatment %This makes things more elegant, albeit sometimes more confusing as well. 
in order to make clear the invariance properties of the expressions.

First we introduce $\mathrm{SU}(3)$ structures, which are a convenient framework for studying nearly K\"ahler manifolds:

\begin{defn}
	An $\mathrm{SU}(3)$ structure $(\omega,\psi = \psi^+ + i \psi^-)$ on a $6$-manifold $M$ is a pair of forms $\omega\in \Omega^2(M)$ and $\psi\in\Omega^3(M,\mathbb C)$ satisfying $$2\omega^3 = 3\psi^+\wedge \psi^-, \qquad \omega\wedge\psi = 0.$$
\end{defn}
We will refer to these equations as the \emph{$SU(3)$ structure equations}.
\begin{theorem}[\cite{carrion1999survey}]
	A nearly K\"ahler structure is equivalent to an $\mathrm{SU}(3)$ structure $(\omega,\psi^+)$ satisfying
	$$d\omega = 3\psi^+, \qquad d\psi^- = -2\omega\wedge\omega.$$
\end{theorem}
We will refer to these equations as the \emph{nearly K\"ahler structure equations}.
\begin{defn}
	A \emph{toric nearly K\"ahler manifold} $(M,\omega,\psi,\mathbb T)$ is a $6$-manifold $M$ equipped with a nearly K\"ahler structure $\mathrm{SU}(3)$ structure $(\omega, \psi = \psi^++i\psi^-)$ which is invariant under the effective action of a $3$-torus $\mathbb T$. We will denote by $(g,J)$ the associated Hermitian structure.
\end{defn}

Let $\mathbb K:\mathfrak t\to \Gamma(TM)$ be the linear map which sends elements of the Lie algebra $\mathfrak t$ of $\mathbb T$ to the vector field generating the corresponding action. The exact forms $d\omega$ and $d\psi^-$ are $\mathbb T$-invariant, so by \cite{madsen2013closed} admit natural multi-moment maps
$$\mu:=\omega\circ(\Lambda^2\mathbb K):M\to\Lambda^2\mathfrak t^*, 
\qquad \varepsilon:=\psi^-\circ(\Lambda^3\mathbb K):M\to\Lambda^3\mathfrak t^*.$$

Define $\theta\in\Omega^1(M,\mathfrak t)$ such that $\theta\circ\mathbb K = \Id$ and $\theta|_{\mathfrak t_{M}^\perp}=0.$
Define $$\gamma:=J\theta = -\theta\circ J\in\Omega^1(M,\mathfrak t)$$
and $\mathring M:=M\backslash\ker\varepsilon.$  Since $\psi$ is a $(3,0)$-form, $\mathring M$ is also described as the set of points where the image $\mathfrak t_M$ of $\mathbb K$ intersects transversally with $J\mathfrak t_M$.
This allows us to write $T\mathring M = \mathfrak t_M\oplus J\mathfrak t_M$ with frame $(\theta,\gamma)\in\Omega^1(\mathfrak t\oplus\mathfrak t)$.

Since $\theta+i\gamma$ gives a framing of $\Lambda^{(1,0)}\mathring M$, we can write $\psi = i\varepsilon \Lambda^3 (\theta+i\gamma)$, so that
$$\psi^+ = \varepsilon(\Lambda^3\gamma-\gamma\wedge\Lambda^2 \theta),
\qquad \psi^- = \varepsilon(\Lambda^3\theta-\theta\wedge\Lambda^2\gamma).$$
	Similarly, the rest of the structures can be given in terms of the multi-moment maps $(\mu,\varepsilon)$, the frame $(\theta,\gamma)$, and a matrix $$C=g\circ(\odot^2\mathbb K):M\to\Sym^2\mathfrak t^*.$$ For example,
	$$\omega = \mu(\Lambda^2\theta +\Lambda^2\gamma) + \theta\wedge \vec c,$$
	where $\vec c = \gamma\lrcorner C \in \Omega^1(\mathring M,\mathfrak t^*).$

\begin{lemma}[\cite{moroianu2018toric}]\label{lemNKEqns}
	Using this framework, the $\mathrm{SU}(3)$ structure equations are equivalent to 
	\begin{equation}\label{eqnToricSU3}
		\det C = \varepsilon^2 + C(V,V),
	\end{equation}
	while the nearly K\"ahler structure equations are equivalent to 
	$$d(\varepsilon\vec c) = 0,
	\qquad \tfrac14\varepsilon d\theta = \vec c\wedge\vec c - \mu \mu\cdot(\gamma\wedge\gamma),$$
	where $V\in \Gamma(\mathring M,\Lambda^3\mathfrak t^*\otimes \mathfrak t)$ is the element corresponding to $\mu$ via the natural isomorphism $\natural:\Lambda^3\mathfrak t^*\otimes\mathfrak t\cong \Lambda^2\mathfrak t^*$.
\end{lemma}
Here, by $\det C\in \Gamma\big(\mathring M,(\Lambda^3\mathfrak t^*)^{\otimes 2}\big)$, we mean the square of the volume form on $\mathfrak t$ induced by $C$. This agrees with the usual determinant in coordinates.

Since the functions $\varepsilon$, $V$, and $C$ are $\mathbb T$-invariant, they descend to $\mathring M/\mathbb T$, which can be locally identified with $\Lambda^2\mathfrak t^*$ via $\mu$. Since we can think of $\mu$ as giving coordinates on $\Lambda^2\mathfrak t^*$, we can think of $\varepsilon$, $V$,  and $C$ as functions locally given in these coordinates on some $U\subset \Lambda^2\mathfrak t^*$.

These coordinates allow explicit computations of several expressions in terms of a potential function:

\begin{theorem}[\cite{moroianu2018toric}]\label{thmInTermsOfPhi}
	There exists a function $\varphi:U\to(\Lambda^3\mathfrak t^*)^2$ whose Hessian in $\mu$ coordinates is $C$. We also have
	$$\varepsilon^2 = \tfrac83(1-\partial_r)\varphi,
	\qquad C(V,V) = (\partial_r^2-\partial_r)\varphi,$$
	where $\partial_r$ is the Euler vector field for $\Lambda^2\mathfrak t^*$ (so that in coordinates $\partial_r = \mu^i\partial_{\mu_i}$).
\end{theorem}

Combining this with equation (\ref{eqnToricSU3})  gives the Monge-Amp\`ere type equation
\begin{equation}\label{eqnNearlyToric}
	\det\Hess\varphi = \left(\tfrac 83-\tfrac{11}3\partial_r+\partial_r^2\right)\varphi,\tag{$\star$}
\end{equation}
which we will refer to as the \emph{toric nearly K\"ahler equation} or just \eqref{eqnNearlyToric}.

Note that with respect to the frame $(\theta,\gamma)$, $g$ is represented by the matrix
$$D:=\begin{pmatrix}
	\Hess\varphi & -\mu \\ \mu & \Hess\varphi
	\end{pmatrix}\in\Gamma\big(U,\mathrm{Sym}^2(\mathfrak t\oplus\mathfrak t)^*\big),$$
where $\mu\in\Gamma(U,\Lambda^2\mathfrak t^*)$ is the inclusion (identity) map.

The above theorem has a partial converse:

\begin{theorem}[\cite{moroianu2018toric}]\label{thmMNInverseConstruction}
	Every solution of the toric nearly K\"ahler equation on some open set $U$ of $\Lambda^2\mathfrak t^*$ defines in a canonical way a nearly K\"ahler structure with $3$ linearly independent commuting infinitesimal automorphisms on $U_0\times\mathbb T^3$, where $$U_0 = \left\{x\in U: (1-\partial_r)\varphi> 0 \text{ and } D \text{ is positive definite}\right\}.$$
\end{theorem}

Note that if $\varphi$ is given by a toric nearly K\"ahler structure, then $(1-\partial_r)\varphi$ is proportional to the  $\varepsilon^2$, and  $D$ is the expression of $g$ in the frame $(\theta,\gamma)$. 
%A careful reading of the proof finds that $U_0$ is assumed to be simply connected. We will work in section \ref{secGlobal} to show that one can always choose a suitable $U$ so that this is true. 
Now consider the following set with an a priori weaker constraint than $U_0$: $$\hat U_0 := \big\{x\in U: (1-\partial_r)\varphi> 0 \text{ and } \Hess\varphi \text{ is positive definite}\big\}.$$
With some linear algebra, we will see that this constraint is actually not weaker:

\begin{lemma}\label{lemEpsDomVol}
	$U_0 = \hat U_0.$
\end{lemma}
\begin{proof}
	Since $D$ being positive definite implies that $\Hess\varphi$ is positive definite, we find that $U_0\subseteq \hat U_0$. It remains to show that $D$ has no null vectors in $\hat U_0$, which implies the reverse inclusion.
	
	Let $C=\Hess\varphi$ and $\varepsilon^2 = \tfrac 83(1-\partial_r)\varphi$. 
	%The computation $C(V,V) = (\partial_r^2-\partial_r)\varphi$ from theorem \ref{thmInTermsOfPhi} applies in this setting, since it didn't rely on the nearly K\"ahler structure.
%	Fix a point $p\in \hat U_0$. To get a contradiction, assume there exists some non-trivial $(v,w)\in \ker D_p\leq \mathfrak t\oplus\mathfrak t$. Thus $$\begin{pmatrix}
%	0 \\ 0
%\end{pmatrix}	 = D\begin{pmatrix}
%	v \\ w
%\end{pmatrix} =  \begin{pmatrix}
%	Cv-\mu w
%\\	\mu v + C w
%\end{pmatrix}.$$
Defining $j = C^{-1}\mu$, we find that any null vector for $D$ is of the form $(v,w)\in \mathfrak t\oplus \mathfrak t$ at some point $p\in \hat U_0$ with 
$$jw = v, \qquad jv = -w.$$
Thus $v$ and $w$ are eigenvectors of $j^2$ at $p$ with eigenvector $-1$. Thus it suffices to show that $j^2\in\Gamma\big(U_0,\End(\mathfrak t)\big)$ never attains an eigenvector $-1$.

Choosing a basis for $\mathfrak t$ so that $C_p$ is diagonal at any chosen $p\in \hat U_0$ allows one to verify that $C^{-1}\mu C^{-1} = \left(\frac{CV}{\det C}\right)^\natural$, where we abuse notation by using $\natural$ to also denote the isomorphism $\Lambda^3\mathfrak t\otimes \mathfrak t^*\cong \Lambda^2\mathfrak t$. Then
$$ j^2 = C^{-1}\mu C^{-1}\mu = \left(\frac{CV}{\det C}\right)^\natural V^\natural = \frac{-C(V,V)\Id + (CV)\otimes V}{\det C},$$
where throughout this computation we've been using juxtaposition to denote `matrix multiplication', or contraction of a single $\mathfrak t, \mathfrak t^*$ index pair. Since $V\in\ker \mu\leq \ker j$, we find that $j^2$ has eigenvalues $0$ with multiplicity $1$ and $-\frac{C(V,V)}{\det C}$ with multiplicity $2$. By the toric nearly K\"ahler equation, $-1$ is an eigenvalue only when $\varepsilon=0$, which is impossible on $\hat U_0$ by definition.
\end{proof}

This lemma can be used to interpret what goes wrong when trying to find a completion of a local toric nearly K\"ahler manifold. If some connected $\mathring M$ is maximal in the sense that it is not properly contained in a toric nearly K\"ahler manifold where $\varepsilon$ doesn't vanish, what is happening at the boundary? Using $\mu$, we can interpret this boundary as a set of points in $\Lambda^2\mathfrak t^*$. By (\ref{eqnToricSU3}), $C$ is going to remain positive definite as long as $\varepsilon$ doesn't vanish. Thus the previous lemma shows that if $\varepsilon$ does not limit to $0$ at the boundary point, then the local solution $\varphi$ to the toric nearly K\"ahler equation cannot be extended to the boundary point. In section \ref{secRadial}, we show that local radial solutions can be extended to have the radius defined between $0$ and some finite $r_0$. The differential equation is singular at $0$, while $\varepsilon$ vanishes when the radius is $r_0$.

%Let $\nu\in\Lambda^3\mathfrak t$ be a fixed volume form. Let $\tilde\varepsilon = \varepsilon\cdot\nu$, and $\tilde C = C\cdot\nu^{\odot 2}$. Note that $\tilde C \in\mathcal C^\infty(M,\Sym^2\Lambda^2\mathfrak t)$.

%Multiplying the first equation in lemma \ref{lemNKEqns} by $\nu^2$ gives
%$$\nu^2\det C = \tilde \varepsilon^2 + \tilde C(\mu,\mu).$$
%Note that the identification $\Lambda^2\mathfrak t*\cong T\Lambda^2\mathfrak t^*$ identifies $\mu$ with $\partial_r$. 

\section{Relation to toric $G_2$ manifolds}

For a strict nearly K\"ahler manifold $(M,\omega,\psi^++i\psi^-)$ with metric $g$, consider the Riemannian cone $\big(N=M\times (0,\infty),g_N = r^2 g + dr^2\big)$, where $r\in(0,\infty)$ is the radial coordinate. It is well known that $N$ admits a $G_2$ structure given by $$\varphi:=d\left(\frac{r^3\omega}3\right),	
\quad *\varphi:=-d\left(\frac{r^4\psi^-}4\right).$$
If $M$ is toric, then the torus action lifts to an multi-Hamiltonian action on $N$ with respect to the form $\varphi$ and $*\varphi$. This makes $N$ a \emph{toric $G_2$} manifold as studied in \cite{madsen2018toric}. The corresponding multi-moment maps for $\varphi$ and $*\varphi$ respectively are 
\begin{align}\label{eqnConeMoments}
	\nu_N:=\tfrac 13r^3\mu:N\to\Lambda^2\mathfrak t^*,\quad \varepsilon_N:=-\tfrac 14r^4\varepsilon:N\to\Lambda^3\mathfrak t^*.
\end{align}

From \cite{madsen2018toric}, $\nu_N\oplus\varepsilon_N$ maps the set of singular orbits $S$ of $N$ to a graph in $\Lambda^2\mathfrak t^*\oplus\Lambda^3\mathfrak t^*\cong\mathbb R^3\oplus\mathbb R$. Moreover, $\varepsilon_N$ is constant on each connected component of $S$. In the case when $N=M\times (0,\infty)$ is the cone over a toric nearly K\"ahler manifold, then the radial symmetries of (\ref{eqnConeMoments}) imply that $\varepsilon_N$ vanishes on the graph, and moreover each edge of the graph is a radial ray shining out from the origin in $\Lambda^2\mathfrak t^*$. Since points on the edge of the graph correspond to torus orbits where a single circle collapses, we immediately find

\begin{cor}\label{corSingDisc}
	On a toric nearly K\"ahler manifold, the torus action is free away from a discrete set of orbits where a single circle collapses and $\varepsilon$ vanishes.
\end{cor}

Since $C$ is the metric on the torus orbits, the vanishing locus of $\det C$ is the set where the torus action is not free. Note that by (\ref{eqnToricSU3}), the positive functions $\varepsilon^2$ and $C(V,V)$ both vanish on singular orbits. Thus $V$ is generates the circle which collapses.

\section{Global properties}\label{secGlobal}

Let $(M,\omega,\psi,\mathbb T)$ be a connected complete toric nearly K\"ahler $6$-manifold. In this section we will prove theorem \ref{thmMain} about the global properties of $M$. Recall that we define $\mathring M = M\backslash\ker\varepsilon$.

\begin{lemma}
	$\mu|_{\mathring M}$ is a submersion.
\end{lemma}
\begin{proof}
	Lemma 4.1(i) in \cite{moroianu2018toric} gives $d\mu|_{\mathring M} = -4\varepsilon\cdot\gamma$. The result follows since $\gamma$ has full rank and $\varepsilon$ does not vanish on $\mathring M$.
\end{proof}

Using that $\varepsilon^2$ is decreasing in radial directions, we can show that $0\in\mu(\mathring M)$:

\begin{lemma}\label{lemStarShape}
	Every $p\in \mathring M$ is contained in some path $\ell$ such that $\mu|_\ell:\ell\to\Lambda^2\mathfrak t^*$ is an injective map whose image is a line segment between $0$ and $\mu(p)$.
\end{lemma}
\begin{proof}
	If $\mu(p)=0$, there is nothing to show. Otherwise, there exists some maximal line segment $L$ contained in $[0,\mu(p)]$ which lifts to a path $\hat L$ in $\mathring M$ containing $p$. By the previous lemma, $L$ is non-empty and open. If $L$ is closed, then $L = [0,\mu(p)]$ as required. Otherwise, since $M$ is complete, $\hat L$ has a limiting point $p'\in M\backslash\mathring M$ so that $L = (\mu(p'),\mu(p)]$. Thus $\varepsilon(p')=0$. But theorem \ref{thmInTermsOfPhi} shows that $\partial_r(\varepsilon^2)=-\tfrac 38 C(V,V)\leq 0$. Since $L$ is in a radial direction, as is the derivative $\partial_r$, we find that $\varepsilon^2$ is negative along $L$, a contradiction.
\end{proof}

	Since $\mu$ and $\varepsilon$ are $\mathbb T$-invariant, they induce maps $\bar\mu:M/\mathbb T\to\Lambda^2\mathfrak t^*$ and $\bar\varepsilon:M/\mathbb T\to\Lambda^3\mathfrak t^3$, which are called \emph{orbital multimoment maps}.

\begin{lemma}
	For any connected component $M_0$ of $\mathring M$, $\bar\mu|_{M_0/\mathbb T}$ is injective.
\end{lemma}
\begin{proof}
	For any $p\in M_0$, let $\ell$ be the path between $p$ and some $p'\in\mu^{-1}(0)$ guaranteed by the previous lemma. The map 
	$$F:M_0/\mathbb T\to\bar\mu^{-1}(0):\mathbb T p\mapsto\mathbb T p'$$
	is clearly well defined and continuous. Since $M_0/\mathbb T$ is connected and $\bar\mu^{-1}(0)$ is discrete, the image of $F$ is a single orbit which we will denote by $o_0\in M_0/\mathbb T$.
	
	Now let $o_1,o_2\in M_0/\mathbb T$ with $\bar\mu(o_1)=\bar\mu(o_2)$.
	 For each $i\in\{1,2\}$, the previous lemma can be used to construct a path $\hat L_i$ between $o_0$ and $o_i$ in $M_0/\mathbb T$ which is a lift of the line segment $L$ between $0$ and $\bar\mu(o_1)=\bar\mu(o_2)$. Since $\mu|_{\mathring M}$ is a submersion, $\bar\mu|_{\mathring M/\mathbb T}$ is a local homeomorphism. In particular, $T_{o_0}\hat L_1 = \bar\mu^{-1}(T_0 L )= T_{o_0}\hat L_2$, implying that $\hat L_1=\hat L_2$. Thus $o_1=o_2$ as required.
\end{proof}

By this lemma, $\mu$ induces global coordinates on $M_0/\mathbb T$. In particular $\varepsilon$ can be viewed as a function on $M_0/\mathbb T$, and $(\mu,\varepsilon)(M_0)$ is the graph of this function. Since $\varepsilon$ vanishes on $\partial M_0$, we find that $M/\mathbb T$ is recovered by gluing together two of these graphs:

	\begin{theorem}\label{thmSphereImage}
		$(\bar\mu,\bar\varepsilon):M/\mathbb T\to\Lambda^2\mathfrak t^*\oplus\Lambda^3\mathfrak t^*$ is injective with image a $3$-sphere.  Moreover, the component of $(\Lambda^2\mathfrak t^*\oplus\Lambda^3\mathfrak t^*)\backslash(\mu,\varepsilon)(M)$ containing $0$ is star-shaped about $0$.
	\end{theorem}
\begin{proof}
	Let $M_+$ be a connected component of $\mathring M$. Since $\varepsilon\neq 0$ on $\mathring M$, $\varepsilon$ has a sign on $M_+$. Without loss of generality, by changing the sign of $\varphi$ if necessary, we can assume that $\varepsilon$ is positive on $M_+$.
	By lemma \ref{lemStarShape}, $U:=\mu(M_+)$ is star shaped around the origin.
	Thus $\partial_r|_{\partial U}$ points outward from the closure $\bar U$. 
	Since $\partial_r(\varepsilon^2)\leq 0$ and $\varepsilon$ vanishes on $\partial U\subseteq\partial\mathring M$, the sign of $\varepsilon$ must change on paths travelling across $\partial U$. Thus there is some other connected component $M_-$ of $\mathring M$ with the opposite sign of $\varepsilon.$ But $\mu(M_-)$ must also be star shaped around $0$ with boundary $\partial U$, so $\mu(M_-)=U=\mu(M_+)$. Since $\partial M_-=\partial M_+$, we have $M=\overline  M_+\cup \overline M_-$. In particular, $\bar\mu:M\to U$ is a double cover ramified over $\partial U$, with the sign of $\bar\varepsilon$ distinguishing the points in each $\bar\mu$ fibre. Thus $(\bar\mu,\bar\varepsilon)$ is injective. Since $U$ is diffeomorphic to a $3$-ball, the image of $(\bar\mu,\bar\varepsilon)$ is a $3$-sphere.
	
	The component of $(\Lambda^2\mathfrak t^*\oplus\Lambda^3\mathfrak t^*)\backslash(\mu,\varepsilon)(M)$ containing $0$ can be written as $D_+\cup D_-$, where $$D_\pm=\bigg\{(x,y)\in\Lambda^2\mathfrak t^*\oplus\Lambda^3\mathfrak t^*:x\in \bar\mu\left(\overline {M_\pm}\right),y\in \big[0,\pm\varepsilon(y)\big]\bigg\}.$$
	Now $D_+\cup D_-$ is star-shaped around $0$ if both $D_\pm$ are. This follows since $\bar\mu(\overline {M_\pm})$ is, and $\pm\varepsilon$ is decreasing in radial directions.
\end{proof}

We can now wrap up the proof of the main theorem:
\begin{proof}[Proof of theorem \ref{thmMain}]
	 The previous theorem combined with corollary \ref{corSingDisc} gives most of the claim. $\mathbb T$ orbits in $\ker\mu$ must be Lagrangian by definition, and there are two of them, since $\bar\mu$ is a double cover ramified at $\ker\varepsilon$.
\end{proof}

\section{Some topology}

We apply the results from the previous section to prove corollary \ref{corNumOrbits}. The obstruction we use to prove this comes from Myer's theorem \cite{myers1941riemannian}, which asserts that if a complete Riemannian manifold has Ricci curvature positive and bounded away from zero, then the diameter must be bounded. Since the same must be true for the universal cover, the fundamental group must be finite. In particular, the first Betti number must vanish. 

\begin{prop}\label{propNotFree}
	Let $(M,\omega,\psi,\mathbb T)$ be a connected complete toric nearly K\"ahler $6$-manifold. Then the action of $\mathbb T$ is not free.
\end{prop} 
\begin{proof}
	 Assume that the action of $\mathbb T$ is free, so that $M$ is a $\mathbb T^3$ bundle over $\mathbb S^3$. It follows that we have the Wang long exact sequence $$H_\bullet(\mathbb T^3)\to H_\bullet(M)\to H_{\bullet-3}(\mathbb T^3)\xrightarrow{[-1]},$$
which shows that $H_1(M)\cong H_1(\mathbb T^3)\cong\mathbb Z^3$ has positive rank. This contradicts Myer's theorem.
\end{proof}

Before we proceed to study the case when the set $S\subset M$ of non-free orbits is not empty, let us introduce some notation. Let $\Lambda < \mathfrak t$ be the lattice of circle subgroups, which allows us to identify $\mathbb T\cong \mathfrak t/\Lambda$. For each orbit $o$ in $S$, there is some $X_o\in\Lambda$ whose induced vector field on $M$ vanishes along $o$. Since $V$ vanishes along $o$, by the definition of $V$, $X_o\propto\mu(o)^\natural$. By lemma \ref{lemStarShape}, $\mu(M)$ is star shaped, so the line through $\mu(o)$ contains two points in $\mu(\ker\varepsilon)=\partial\mu(M)$, which are $o$ and an 'antipodal' point $o'$. Thus $X_o$ can only vanish at $o$ and perhaps also at $o'$.

Consider the case where $S$ has $k$ orbits. We will consider decompositions $M = A\cup B$, where $A$ and $B$ are both unions of $\mathbb T$-orbits and $A\cap B\cap S=\emptyset$. To see that such a decomposition exists, note that if $H$ is a hyperplane in $\Lambda^2\mathfrak t^*\oplus\Lambda^3\mathfrak t^*$ disjoint from $(\mu,\varepsilon)(S)$, then there exists a neighbourhood $U$ of $H$ also disjoint from $(\mu,\varepsilon)(S)$. Now it is clear that we can find $A$ and $B$ as claimed with $(\mu,\varepsilon)(A\cap B) = U$. Moreover, we see that no two orbits in $S\cap A$ (respectively $S\cap B$) correspond to the same element of $\Lambda$, since by the previous paragraph, they would correspond to antipodal points in $\mu(\ker\varepsilon)$, which are avoided by this construction.

Now we have $A\cong D_i$ and $B\cong D_j$ where $i+j = k$ and $D_i$ is a $\mathbb T^3$ fibration over the three-ball $\mathbb D^3$ with $i$ orbits where circles collapse. Moreover, these circles are different, in the sense that the collapsing directions correspond to different vectors in $\mathfrak t$. 

\begin{lemma}\label{lemDiHomlogy}
	$D_0\simeq \mathbb T^3$, $D_1\simeq\mathbb T^2$, $D_2\simeq \mathbb S^3\times\mathbb S^1$.
\end{lemma}
\begin{proof}
Since every bundle over $\mathbb D^3$ is trivial, $D_0 \cong \mathbb T^3\times \mathbb D^3\simeq \mathbb T^3$. 

$D_1$ is a neighbourhood of the collapsed orbit. Thus $D_1\cong \mathbb T^2\times\mathbb R^4$ by identifying $\mathbb R^4$ with $D_1/\mathbb T$ and $\mathbb T^2$ the quotient of $\mathbb T$ with the circle that collapses.

$D_2/\mathbb T$ is a neighbourhood of a curve $C$ connecting the two collapsing orbits. Thus $D_2$ retracts to some $\tilde D_2$ such that $\tilde D_2/\mathbb T\cong C$. Since the circles that collapse are different, we can write $\tilde D_2\cong(\mathbb S^1\times\mathbb S^1\times [0,1])/\sim,$ where $\sim$ collapses the first circle at $0$ and the second circle at $1$. Now $$\tilde D_2\to\mathbb C^2\times\mathbb S^1:[\theta_1,\theta_2,\theta_3,x]\mapsto 
	\left(\sin\left(\tfrac\pi2x \right) e^{i\theta_1},
		\cos\big(\tfrac\pi2 x\right) e^{i\theta_2},\theta_3\big)$$
	identifies $\tilde D_2\cong\mathbb S^3\times\mathbb S^1$.
\end{proof}

Before we proceed to applying Meyer-Vietoris to the decomposition $M=A\cup B$, we still need to understand the equatorial region $E:=A\cap B$. 
\begin{lemma}
	$h_1(E)=h_4(E)\in [2,3]$.
\end{lemma}
\begin{proof}
	$E$ must be a $\mathbb T^3$ bundle over $U\simeq\mathbb S^2$.  
	Thus $E$ retracts to a $\mathbb T^3$ bundle $\bar E$ over $\mathbb S^2$. Since $\bar E$ is compact, we have the duality $h_1(\bar E) =h_4(\bar E)$.
	Part of the Wang sequence is
	$$\to H_{2-2}(\mathbb T^3)\to H_1(\mathbb T^3)\to H_1(\bar E)\to 0.$$
	Thus $h_1(E) = h_1(\bar E)\in h_1(\mathbb T^3) - \big[0,h_0(\mathbb T^3)\big] = 3 - [0,1] = [2,3]$.
\end{proof}

We can now work with the Meyer-Vietoris sequence with respect to the decomposition $M = A\cup B = D_i\cup D_j$. Since $A\cap B = E$, this sequence is 
$$H_\bullet(E)\to H_\bullet(D_i)\oplus H_\bullet(D_j)\to H_\bullet(M)\xrightarrow{[-1]}.$$
We are now ready to prove the main result of this section:

\begin{proof}[Proof of corollary \ref{corNumOrbits}]
	The Meyer-Vietoris sequence at $\bullet = 1$ gives $$h_1(D_i)+h_1(D_j)\leq h_1(E)+h_1(M)\leq 3,$$ where the second inequality uses the previous lemma and Meyer's theorem. But by lemma \ref{lemDiHomlogy}, $h_1(D_i)=3-i$ or $i\in\{0,1,2\}$. In particular, for $k=i+j<3$, we have $h_1(D_i)+h_1(D_j) = 6-(i+j) > 3$, contradicting our upper bound.
	
	For $k=3$, choose $A\cong D_1$ and $B\cong D_2$. Since $h_1(M)=0$, the Meyer-Vietoris sequence at $\bullet = 1$ gives $h_1(E)\geq h_1(D_1)+h_1(D_2) = 3$. Combining this with the previous lemma gives $h_1(E)=h_4(E) = 3$. Since $h_5(M)=h_1(M)=0$, the Meyer-Vietoris sequence at $\bullet = 4$ gives the contradiction $$3 = h_4(E) \leq h_4(D_1)+h_4(D_2) = 1.$$
\end{proof}

\section{Radial solutions}\label{secRadial}
In this section, we study solutions of the form $\varphi(\mu) = x(t)$, where $t  = \tfrac12\|\mu\|^2$ is a radial coodinate, and $\|\cdot\|$ is the Euclidean metric on $\Lambda^2\mathfrak t^*$. These were studied in \cite{moroianu2018toric}, where they show that the nearly toric equation simplifies to the ODE $$ 0 = \mathcal D(x):=3(x'^2-2t)(x'+2tx'')-8(x-2tx')$$ subject to the constraint
\begin{align}\label{eqnRadCon}
	x>2tx'>2t\sqrt{2t},
\end{align}
 where the derivatives are taken with respect to $t$. The main result is that such a radial solution cannot be complete:
 
\begin{theorem}\label{thmNoRadial}
	If $(M,\omega,\psi,\mathbb T)$ is a connected complete toric nearly K\"ahler $6$-manifold corresponding to a solution $\varphi$ to the toric nearly K\"ahler equation, then $\varphi$ is not radially symmetric.
\end{theorem}
\begin{proof}
	Assume that $\varphi$ is radially symmetric. Combining this symmetry with theorem \ref{thmSphereImage}, $\mu(M)$ must be a closed $3$-disc $\Delta$ centred at the origin. Now consider the set of points $S$ in $M$ where the torus action is not free. By corollary \ref{corSingDisc}, $\mu(S)$ is a discrete set of points in $\partial\Delta$. But by radially symmetry, $\mu(S)$ must be either empty or all of $\partial\Delta$. But $\mu(S)$ is a discrete set, so it can't be $\partial\Delta$. Thus $\mu(S)$, and hence $S$ is empty. This contradicts proposition \ref{propNotFree}.
\end{proof} 

We now investigate what goes wrong with the ODE to prevent completeness. Local existence of solutions to ODE's will give a local solution $x(t)$ to $\mathcal D(x)=0$ near any prescribed initial $1$-jet $\big(t_0,x(t_0),x'(t_0)\big)$ satisfying the constraints (\ref{eqnRadCon}). Let $(t_-,t_+)$ be the maximal open interval on which the solution can be extended while satisfying the constraints. 

$t_\pm$ must be either a point where $x(t)$ blows up or a boundary point of the constraints. By lemma \ref{lemEpsDomVol}, $\varepsilon^2 = \tfrac 83(x-2tx')>0$ implies the other constraint $x'>\sqrt{2t}$. Thus the boundary condition is simply $\varepsilon^2=0$.

\begin{lemma}\label{lemBigRadLim}
	$x(t)$ does not blow up at $t_+ < \infty$.
\end{lemma}
\begin{proof}
	First note that $\varepsilon^2>0$ implies that $(\log x)'<\frac 1{2t}$. Integrating this implies that $x(t) < x_0\sqrt{\tfrac t{t_0}}$. Since $x(t)$ is also positive, it cannot blow up in finite time.
	
	On the other hand, integrating $x'>\sqrt{2t}$  gives $x-x_0 > \frac{\sqrt{2t}^3-\sqrt{2t_0}^3}3$. This lower bound for $x$ grows faster as $t$ increases than the upper bound for $x$ in the previous paragraph. Thus $t_+$ is finite.
\end{proof}

We compute
$$\varepsilon^2 = \tfrac 83 (x-2t x'), \qquad 2tx'' = \frac{\varepsilon^2}{x'^2-2t}-x',
\qquad (\varepsilon^2)' = -\frac83\frac{\varepsilon^2}{x'^2-2t}.$$
Note that by lemma \ref{lemEpsDomVol}, the constraints can be rewritten as $0<\varepsilon^2\propto x-2tx'$.

Thus the constraints imply that $x'^2-2t>0$, so the ODE is regular when the constraints hold and $t>0$.

\begin{lemma}
	$t_-=0$.
\end{lemma}
\begin{proof}
	Since $t_-$ is the boundary point of a maximal domain of an ODE subject to the constraint $\varepsilon^2>0$, at $t_-$ either the ODE is singular, the solution $x(t)$ becomes unbounded, or $\varepsilon^2$ vanishes.
	Since the ODE is singular at $t=0$, we must have $t_-\geq 0$. Since $x$ is positive and increasing, it must be bounded in $(t_-,t_+)$. Since $\varepsilon^2$ is decreasing, it cannot vanish at $t_-$. Thus $t_-$ must a singular point of the ODE, in particular the only one: $0$.
\end{proof}

By theorem \ref{thmNoRadial}, there must be some singularity for $x(t)$ in $[0,t_+]$, and by the previous two lemmas it must be at $t=0$.

Note that the estimate in lemma \ref{lemBigRadLim} doesn't essentially require radial symmetry: it only uses $\varepsilon^2\propto \varphi-\partial_r\varphi>0$. In particular, continuing the discussion following lemma \ref{lemEpsDomVol}, $\varphi$ should not become unbounded as one tries to extend solutions in radial directions away from the origin.

\section{Polynomial solutions}
In this section we will try to understand polynomial solutions to the toric nearly K\"ahler equation.
As described in \cite{moroianu2018toric}, the toric nearly K\"ahler structure on $\mathbb S^3\times\mathbb S^3$ corresponds to the solution of the toric nearly K\"ahler equation
\begin{align}\label{eqnCubicSoln}
	\varphi_0 := 3 + \sum_j \mu_j^2 + \frac1{\sqrt 3}\prod_j\mu_j,
\end{align}
where $\{\mu_j\}_{j=1}^3$ are coordinates on $\Lambda^2\mathfrak t^*$ induced by the multi-moment map $\mu$. We will prove theorem \ref{thmPolySol} by treating each degree of polynomial separately. First we will introduce some notation. If $E$ is an equation or expression, and $m$ is a monomial in $\mathbb R[\mu_1,\mu_2,\mu_3],$ then $[m]E$ and $(m)E$ will refer respectively to the coefficient of $m$ in $E$, and the part of $E$ which is a multiple of $m$. We will use $\nabla$ to denote the gradient in $\{\mu_j\}_{j=1}^3$ coordinates, and abuse notation by not distinguishing it from its transpose, or a restricted gradient to an context-appropriate subset of the coordinates. Similarly $\nabla^2$ will denote the Hessian, where the set of coordinates may depend on context.

\begin{prop}
	Every cubic solution to the toric nearly K\"ahler equation is equivalent to $\varphi_0$ up to linear changes in coordinates.
\end{prop}
\begin{proof}
	Let $\varphi$ be some cubic solution of the toric nearly K\"ahler equation \eqref{eqnNearlyToric}. Write $\varphi= \sum_{j=0}^3\varphi^k$, where each $\varphi^k$ is a degree $k$ homogeneous polynomial in $\{\mu_j\}_{j=1}^3$. As noted in \cite{moroianu2018toric}, $\varphi^1$ can be chosen to be $0$. The degree $3$ term of \eqref{eqnNearlyToric} gives
	$$|\nabla^2\varphi^3|
	= \left(\tfrac 83-\tfrac{11}3\partial_r+\partial_r^2\right)\varphi^3 
	= \left(\tfrac 83-{11}+9\right)\varphi^3 =\tfrac 23\varphi^3.$$
	Since $|\nabla^2\varphi^3|\propto\varphi^3$, the plane algebraic curve $V(\varphi^3)\subset \mathbb{CP}^2$ is a union of lines \cite{Steiger2008Inflexion}, so that $\varphi^3$ is a product of linear factors. We can choose coordinates along these lines so that $\varphi^3=\lambda\prod_{j=1}^3\mu_j$. We can compute $|\nabla^2\varphi^3|=2\lambda^2\varphi^3,$ so we must have $\lambda^2 = \tfrac 13$. Again by changing coordinates, we may assume that $\lambda = \tfrac 1{\sqrt 3}$.
	
	Writing $\varphi^2 = \varphi^2_d + \varphi^2_o$, where $\nabla^2\varphi^2_d$ is a diagonal matrix, while $\nabla^2\varphi^2_o$ vanishes on the diagonal, one can compute that the degree $2$ term of $|\nabla^2\varphi|$ is $\frac23(\varphi^2_o-\varphi^2_d)$.
	Thus the degree $2$ term of (\ref{eqnNearlyToric}) is 
	$$-\frac23\varphi^2 = \frac23(\varphi^2_o-\varphi_d),$$
	so that $\varphi_o=0$. Thus $\nabla^2\varphi^2$ is a diagonal matrix. We still have the freedom to scale the coordinates so that $\nabla^2\varphi^2=2\Id$, or equivalently $\varphi^2=\sum_{j=1}^3\mu_j^2.$ The degree $0$ term of \eqref{eqnNearlyToric} now gives $\tfrac83\varphi^0=|\nabla^2\varphi^2|=8$, so that $\varphi^0=3$. Thus $\varphi=\varphi_0$.
\end{proof}

For higher degree polynomial solutions, the toric nearly equation becomes overdetermined, since $|\nabla^2\varphi|$ is formally a polynomial of degree $\\3(\deg(\varphi)-2)$. This suggests that the cubic solution might be the only polynomial solution. By diagonalizing $\nabla^2\varphi^2$, we can always choose a basis so that $\varphi^2=\sum_{j=1}^3\mu_j^2$. As discussed in the previous proof, we will also get $\varphi^0=3$ and $\varphi^1=0$. The degree $1$ term of (\ref{eqnNearlyToric}) tells us that $\varphi^3$ is harmonic.

Our main tool will be the following theorem, which we will refer to as Hesse's theorem since Hesse originally claimed the result:
\begin{theorem}[\cite{lossen2004does}]
	If $f$ is a homogeneous polynomial in $4$ or less variables over an algebraically closed field, then $\det\circ\Hess(f)=0$ if and only if $f$ is independent of one of the variables after a suitable homogeneous coordinate change.
\end{theorem}
	We continue proving our theorem with the quartic case:

\begin{prop}
	There are no quartic solutions to the toric nearly K\"ahler equation.
\end{prop}
\begin{proof}
Let $\varphi=\sum_{j=0}^4\varphi^j$ be a solution to \eqref{eqnNearlyToric}, where each $\varphi^j$ is a homogeneous polynomial of degree $j$. Assume that $\varphi^4\neq 0$, so that $\varphi$ is quartic. We can choose coordinates so that the quadratic part of $\varphi$ is $3+x^2+y^2+z^2$, where $(x,y,z)$ is a relabelling of the coordinates $(\mu_j)_{j=1}^3$.

The degree $6$ term of \eqref{eqnNearlyToric} is $|\nabla^2\varphi_4|=0$. By Hesse's theorem, this is equivalent to $\varphi_4$ being a function of two variables. Thus, we can choose coordinates so that $\varphi^4 = \varphi^4(x,y)$.

\begin{lemma}$\varphi^3_{zz}=0$.
\end{lemma}
\begin{proof}
Assume that $\varphi^3_{zz}\neq 0$. The degree $5$ term of \eqref{eqnNearlyToric} gives $|\nabla^2_{xy}\varphi^4|\varphi^3_{zz}=0$, so that $|\nabla^2_{xy}\varphi^4|=0$.
Using Hesse's theorem, we can choose coordinates so that $\varphi^4$ depends only on $x$. Now the degree $4$ term of \eqref{eqnNearlyToric} gives $$4\varphi^4=\varphi^4_{xx}|\nabla^2_{yz}\varphi^3|.$$
Write $\varphi^3 = B^0x^3 + B^1x^2+B^2x+B^3$, where each $B^j$ is a degree $j$ homogeneous polynomial in $y$ and $z$. The $x^2$ term of the above equation gives $0=|\nabla^2 B^3|$. By Hesse's theorem, we can choose coordinates so that $B^3\propto y^3$. Note that we still can assume that $\varphi^3_{zz}\neq 0$, since $z$ is still a coordinate orthogonal to $x$ with respect to $\nabla^2\varphi^2$. Since $\varphi^3_{zz}\neq0$, we must have $B^2_{zz}\neq 0$. The $x^3$ term of the above equation then gives $B^3_{yy}B^2_{zz}=0$, implying $B^3=0$. The $x^4$ term of the above equation gives $|\nabla^2B^2| = \tfrac 13$.

  Now $\Delta\varphi^3=0$ implies that $B^1=0$. Now the degree $3$ term of $|\nabla^2\varphi|$ modulo $x^3$ is $$|\nabla^2(xB^2)| = -2xB^2|\nabla^2 B^2| = -\tfrac 23 xB^2,$$
  where the first equality can be verified directly since $B^2$ is a homogeneous quadratic polynomial.
  Thus the coefficient of $x$ of the degree $3$ term of (\ref{eqnNearlyToric}) is $\tfrac 23 B^2 = -\tfrac23B^2$, so that $B^2=0$.
  This contradicts $|\nabla^2B^2| = \tfrac 13$.
\end{proof}
Since $\varphi^3_{zz}=0$, we can write $\varphi^3 = B^3(x,y) + zB^2(x,y).$ We have
$$\varphi = 3 + x^2+y^2+z^2 + B^3 + zB^2 + \varphi^4.$$
Since $\varphi^3$ is harmonic, so are $B^2$ and $B^3$.
$$\nabla^2\varphi = \begin{pmatrix}
	2\Id + \nabla^2(B_3+\varphi^4)	& \nabla B^2
\\	\nabla B^2	& 2
\end{pmatrix}+z\begin{pmatrix}
	\nabla^2 B^2	& 0
\\	0	& 0
\end{pmatrix}.$$
Now $[z^2]$\eqref{eqnNearlyToric} gives
$2|\nabla^2 B^2| =-\tfrac23.$
Combining this with $\Delta B^2=0$, we can rotate coordinates so that $B^2 = \tfrac{xy}{\sqrt 3}$.
Now $(z)$\eqref{eqnNearlyToric} gives
$$0 = \tfrac 4{\sqrt 3}(B^3+\varphi^4)_{xy}.$$
Thus $0 = \varphi^4_{xy} = B^3_{xy}$. Since $B^3$ is harmonic, it must vanish. Note that we now have $\varphi = \varphi_0 + \varphi^4.$ Now it is easy to see that the degree $2$ part of \eqref{eqnNearlyToric} gives $\Delta\varphi^4=0$, since it is the only term depending on $\varphi^4$. Combining this with $\varphi^4_{xy}=0$ shows that $\varphi^4=0$, a contradiction.
\end{proof}

Working quite a bit harder we can establish the quintic case:

\begin{prop}
	There are no quintic solutions to the toric nearly K\"ahler equation.
\end{prop}
\begin{proof}
	Let $\varphi=\sum_{j=0}^5\varphi^j$ be a solution to \eqref{eqnNearlyToric} with $\varphi^5\neq 0$.
	The degree $9$ term of \eqref{eqnNearlyToric} gives $|\nabla^2\varphi^5|=0$. Using Hesse's theorem, we can change variables so that $\varphi^5_z=0$. 
	\begin{lemma}
		$\varphi^4_{zz}=0$
	\end{lemma}
	\begin{proof}
		Assume $\varphi^4_{zz}\neq 0$. The degree $8$ term of \eqref{eqnNearlyToric} gives  $0=\varphi^4_{zz}|\nabla^2_{xy}\varphi^5|.$ Thus $|\nabla^2_{xy}\varphi^5|=0$. By Hesse's theorem, we may assume that $\varphi^5_y=0$. Write $\varphi^4 = \sum_j C^j x^{4-j}$, where each $C^j$ is a degree $j$ homogeneous polynomial in $y$ and $z$. The degree $7$ term of \eqref{eqnNearlyToric} gives 
		\begin{align*}
			0 =& |\nabla^2_{yz}\varphi^4|= |\nabla^2 C^4 + x\nabla^2 C^3 + x^2 \nabla^2 C^2|
			\\	=& |\nabla^2 C^4| + x\langle \nabla^2 C^4,\nabla^2 C^3\rangle 
			\\	&+ x^2\big(|\nabla^2 C^3| +\langle \nabla^2 C^4,\nabla^2 C^2\rangle\big)
			\\	&+ x^3 \langle \nabla^2 C^3,\nabla^2 C^2\rangle + x^4|\nabla^2 C^2|,
		\end{align*} 
		where $\langle N,M\rangle:= 2\adj(N)\cdot M = [t]|N+tM|.$
		If any $C^j$ vanishes, then we have $0=|\nabla^2 C^k|=|\nabla^2 C^\ell|=\langle\nabla^2 C^k,\nabla^2 C^\ell\rangle$, where $\{j,k,\ell\}=\{2,3,4\}$. This implies that $C^k$ and $C^\ell$ are both proportional to powers of the same linear term, which we may choose to be $y$ by changing coordinates. This contradicts $\varphi^4_{zz}\neq 0$.
		Thus no $C^j$ vanishes. Since $|\nabla^2C^2|=0$, we may use Hesse's theorem to choose coordinates so that $C^2\propto y^2$. Since $\langle \nabla^2 C^3,\nabla^2C^2\rangle=0$, we must have $C^3_{zz}=0$. Thus $C^4_{zz}=\varphi^4_{zz}\neq 0$. Combining this with $|\nabla^2C^4|=0$, we may choose coordinates so that $C^4\propto z^4$. Then $0=\langle\nabla^2 C^4,\nabla^2 C^3\rangle$ implies that $C^3_{yy}=0$. Thus  $C^3= 0$, a contradiction
	\end{proof}

\begin{lemma}$\varphi_{zzz}=0$.\end{lemma}
\begin{proof}
	Assume $\varphi_{zzz}\neq 0$. 
	Write $\varphi = \alpha + \beta z + \gamma z^2 + \delta z^3$, where $\alpha,\beta,\gamma,$ and $\delta$ are polynomials in $x$ and $y$ of degrees  $5,3,1,$ and $0$ respectively. Again, we will use exponents to denote the corresponding homogeneous parts.  We are assuming $\delta\neq 0$. 
	We have 
	$$\nabla^2\varphi = \begin{pmatrix}
		\nabla^2\alpha	& \nabla\beta
	\\	\nabla\beta	& 2\gamma
	\end{pmatrix}+z\begin{pmatrix}
		\nabla^2\beta	& 2\nabla\gamma
	\\	2\nabla\gamma		& 6\delta
	\end{pmatrix}.$$
	The degree $7$ term of \eqref{eqnNearlyToric}  with a factor of $z$ gives $|\nabla^2\alpha^5|=0$. The degree $6$ term  with a factor $z^2$ gives $\langle\nabla^2\alpha^5,\nabla^2\beta^3\rangle=0$. The degree $5$ term with a factor of $z^3$ gives $|\nabla^2\beta^3|=0$. Hesse's theorem allows us to interpret this as $\alpha^5$ and $\beta^3$ each depending on only one variable, which must be the same due to the cross-term. Thus we can assume that $\alpha^5$ and $\beta^3$ are both functions of $x$.
	
	The degree $6$ term with a factor of $z$ gives $0=\alpha^5_{xx}\alpha^4_{yy}\delta$, so that $\alpha^4_{yy}=0$. Thus we have $\nabla^2_{yz}\varphi^4=0$. Note that we now have no distinguished direction in the $y-z$ plane.
	
	We write $\varphi = A + B + C + D$, where $A,B,C,$ and $D$ have degree in $[y,z]$ respectively $0, 1, 2,$ and $3$ and total degrees respectively $5,4,3$ and $3$. The degree $5$ term of \eqref{eqnNearlyToric} with a factor of $x^3$ gives $0 = A^5_{xx}|\nabla^2_{yz} D^3|$. Since $A^5=\varphi^5\neq 0$, this shows that $|\nabla^2_{yz}D^3|=0$. Using Hesse's theorem, we can write $D^3$ as a function of $y$. This contradicts $\varphi_{zzz}\neq 0$.
\end{proof}

We will continue to use the decomposition with Greek letters. The previous lemma shows that $\delta=0$.

\begin{lemma}
	$\gamma = 1$.
\end{lemma}
\begin{proof}
 	Assume $\gamma\neq 1$, so that $\nabla\gamma\neq 0$. We can rotate the $x$-$y$ coordinates to write $\gamma=1 + c x$ for some $0\neq c\in\mathbb R$. Thus  $(z^3)$\eqref{eqnNearlyToric} gives $0=-\beta_{yy}(2c)^2$, so that $\beta_{yy}=0$.
	 We compute
	$$(z^2)|\nabla^2\varphi|=-4c^2\alpha_{yy}-2\gamma(\beta_{xy})^2+4c\beta_{xy}\beta_y
	=-4c^2\alpha_{yy}-2(\beta_{xy})^2+2c\beta_{xy}(2-\partial_r)\beta_y.$$
	Thus $(z^2y)$\eqref{eqnNearlyToric} gives $\alpha_{yyy}=0$. 
	Now 	$[z^2x^3]$\eqref{eqnNearlyToric} gives $\alpha_{yy}^5=0$.	
	We compute
	$$(yz)|\nabla^2\varphi| = 2\begin{vmatrix}
		0	& \alpha_{xyy}	& \beta_{xy}
	\\	\beta_{xy}	& \alpha_{yy}	& \beta_y
	\\	2c	&	\beta_y	&	2\gamma
	\end{vmatrix}.$$
	This is a polynomial in $x$, whose quartic term is proportional to $(\beta_{xxy}^3)^3$. Thus $[x^4yz]$\eqref{eqnNearlyToric} gives $\beta^3_y=0$.
	Now $[x^2z^2]$\eqref{eqnNearlyToric} gives $\alpha^4_{yy}=0$. Now the degree $7$ term of \eqref{eqnNearlyToric} gives $0=c(\alpha^5_{xy})^2$, so that $\alpha^5_y=0$.
	
	Now $[z^2]$\eqref{eqnNearlyToric} gives $\tfrac23 = \beta^2_{xy}+c^2$, so that $[xz^2]\eqref{eqnNearlyToric}$ gives $0=\alpha^3_{xyy}+c$. But $\varphi$ is harmonic, giving
	$$0 = \Delta\varphi^3 = \Delta\alpha^3 + 2c + z\Delta\beta^2 = \alpha^3_{xx}+z\beta^2_{xx}.$$
	
	Thus $\alpha^3_{xx}=0=\beta^2_{xx}$. Now $[x^3]$\eqref{eqnNearlyToric}  gives $\alpha^5=0$, contradicting $\varphi^5\neq 0$.
\end{proof}

Now that $\gamma=1$, $(z^2)$\eqref{eqnNearlyToric}  gives $2|\nabla^2\beta|=-\sfrac 23$. Using Hesse's theorem, we can change coordinates so that $\beta^3= bx^3$, for some $b\in\mathbb R$. Combining these equations with $\Delta\beta^2=0$ (from $\nabla\varphi^3=0$), we can choose coordinates so that $\beta^2=\frac{xy}{\sqrt 3}$, independent of whether or not $b$ vanishes. We will show that $b$ does indeed vanish. $(z)$\eqref{eqnNearlyToric} gives
\begin{align*}
	\frac{2xyz}{3\sqrt 3} + 4bzx^3 =& [z]|\nabla^2\varphi| = z\left(2\langle\nabla^2\alpha,\nabla^2\beta\rangle+\begin{vmatrix}
		\nabla^2\beta	& \nabla\beta
	\\	\nabla\beta		& 0
	\end{vmatrix}\right)
\\	=& z\left(12bx\alpha_{yy}-\frac4{\sqrt 3}\alpha_{xy}-6bx(\beta_y)^2+\frac2{\sqrt 3}\beta_x\beta_y\right)
\\	=& z\left(12bx\alpha_{yy}-\frac4{\sqrt 3}\alpha_{xy}-2bx^3+\frac23x\left(3bx^2+\frac y{\sqrt 3}\right)\right),
\end{align*} 
so that
$$0 = 12bx\alpha_{yy} - \frac4{\sqrt 3}\alpha_{xy}-4bx^3.$$
Taking the coefficients of $x^0$, $x^1$, and $x^2$ respectively of this equation gives
$$	[x^1]\alpha_y=0, 
\qquad [x^2]\alpha_y= \frac{3\sqrt 3}2b[x^0]\alpha_{yy},
\qquad [x^3]\alpha_y=\sqrt 3 b[x^1]\alpha_{yy}.
$$
In particular, $6\sqrt 3 b =\alpha^3_{xxy} =  - \alpha^3_{yyy},$ where the second inequality comes from $\alpha^3$ being harmonic.
Now we compute
$$	[x^0z^0]|\nabla^2\varphi| = [x^0]\begin{vmatrix}
	\alpha_{xx}	& 0	& \frac y{\sqrt 3}
\\	0	& \alpha_{yy}	& 0
\\	\frac y{\sqrt 3}	& 0	& 2
\end{vmatrix} = \alpha_{yy}\left(2\alpha_{xx}-\frac{y^2}3\right).
$$
In particular, $[y^2]$\eqref{eqnNearlyToric} gives
$$0 = [y^2]\alpha_{yy}\alpha_{xx}=\alpha^4_{yyyy}+\alpha_{yyy}\alpha_{xxy} +\alpha_{xxyy} = \alpha^4_{yyyy} - 162b^2.$$
Now $[x^3yz]$\eqref{eqnNearlyToric} gives 
$$0 = b\alpha_{xxyyy}\propto b^2\alpha^4_{yyyy}\propto b^4,$$
so that $b=0$. Now $[z]\eqref{eqnNearlyToric}$ is simply $\alpha_{xy}=0$. The remainder of \eqref{eqnNearlyToric} is
$$\left(\tfrac 83-\tfrac{11}3\partial_r+\partial_r^2\right)\alpha =  \begin{vmatrix}
	\alpha_{xx}	& 0	& \frac y{\sqrt 3}
\\	0	& \alpha_{yy}	& \frac x{\sqrt 3}
\\	\frac y{\sqrt 3}	& \frac x{\sqrt 3}	& 2
\end{vmatrix} = 2\alpha_{xx}\alpha_{yy} + \frac 13(\partial_r-\partial_r^2)\alpha.$$
Thus $(xy)$\eqref{eqnNearlyToric} gives $0=\alpha_{xxx}\alpha_{yyy}$. Without loss of generality, by changing coordinates we have $\alpha_{yyy}=0$. Thus \eqref{eqnNearlyToric} becomes
$$0 = \left(\tfrac 83-4\partial_r+\tfrac 43\partial_r^2\right)\alpha-4\alpha_{xx}.$$
In particular, $0 = -\frac{32}3\alpha^5$, contradicting $\varphi^5=\alpha^5\neq 0$.
\end{proof}

Combining the previous three propositions gives theorem \ref{thmPolySol}.

\bibliography{myBib}{}
\bibliographystyle{plain}
\end{document}